\documentclass{amsart}

\usepackage{amsmath,amsfonts,amssymb,amsthm}
\usepackage{tikz}
\usepackage{caption}
\usepackage{float}

\newtheorem{theorem}{Theorem}[section]

\newtheorem{lemma}[theorem]{Lemma}
\newtheorem{cor}[theorem]{Corollary}

\theoremstyle{definition}

\theoremstyle{remark}

\numberwithin{equation}{section}

\begin{document}

\title{Lattices in Chip-Firing}

\author{Patrick Liscio}

\begin{abstract}

  We analyze the poset of moves in chip-firing, as defined by Klivans and Liscio.  Answering a question of Propp, we show that the move poset forms the join-irreducibles of the poset of configurations.  The proof involves a graph augmentation and an analysis of configurations in which only one firing move is available.  We then use this framework to analyze the problem of chip-firing on a line, where the move poset is relevant to the problem of labeled chip-firing.

\end{abstract}

\maketitle

\section{Introduction}

In a chip-firing process, a collection of indistinguishable chips are placed at the nodes of a graph.  If a node has at least as many chips as it has neighbors, it can ``fire'' by sending one chip to each of its neighbors.  The process terminates if no site has enough chips to fire.

We define a \textbf{chip configuration} on a graph $G$ to be a $\mathbb{N}$-valued function on the vertices of $G$, designating the number of chips at each vertex of $G$.  Given a particular configuration $c$ for which the chip-firing process beginning at $c$ terminates, we can define a $\textbf{configuration poset}$ beginning at $c$ as follows.  Consider every configuration that is reachable from $c$ after some sequence of firing moves.  For two such configurations $c_1$ and $c_2$, we say that $c_1\ge c_2$ if it is possible to go from $c_1$ to $c_2$ through some sequence of firing moves.  Note that the direction of this ordering is reversed from the usual convention, in which later moves are greater than earlier ones.

We then define the \textbf{move poset} for an initial configuration $c$.  It is known that any sequence of firing moves from $c$ to completion must consist of a fixed number of firing moves at each site $k$.  If we define $k^j$ to be the $j^{th}$ firing move at site $k$, then we define $k^j \ge k'^{j'}$ if move $k^{j}$ cannot occur after $k'^{j'}$.  The main result of this paper relates these two posets:

\begin{theorem}
The join-irreducibles of the configuration poset form the move poset.
\end{theorem}

This provides a justification for using the move poset in order to analyze chip configurations.

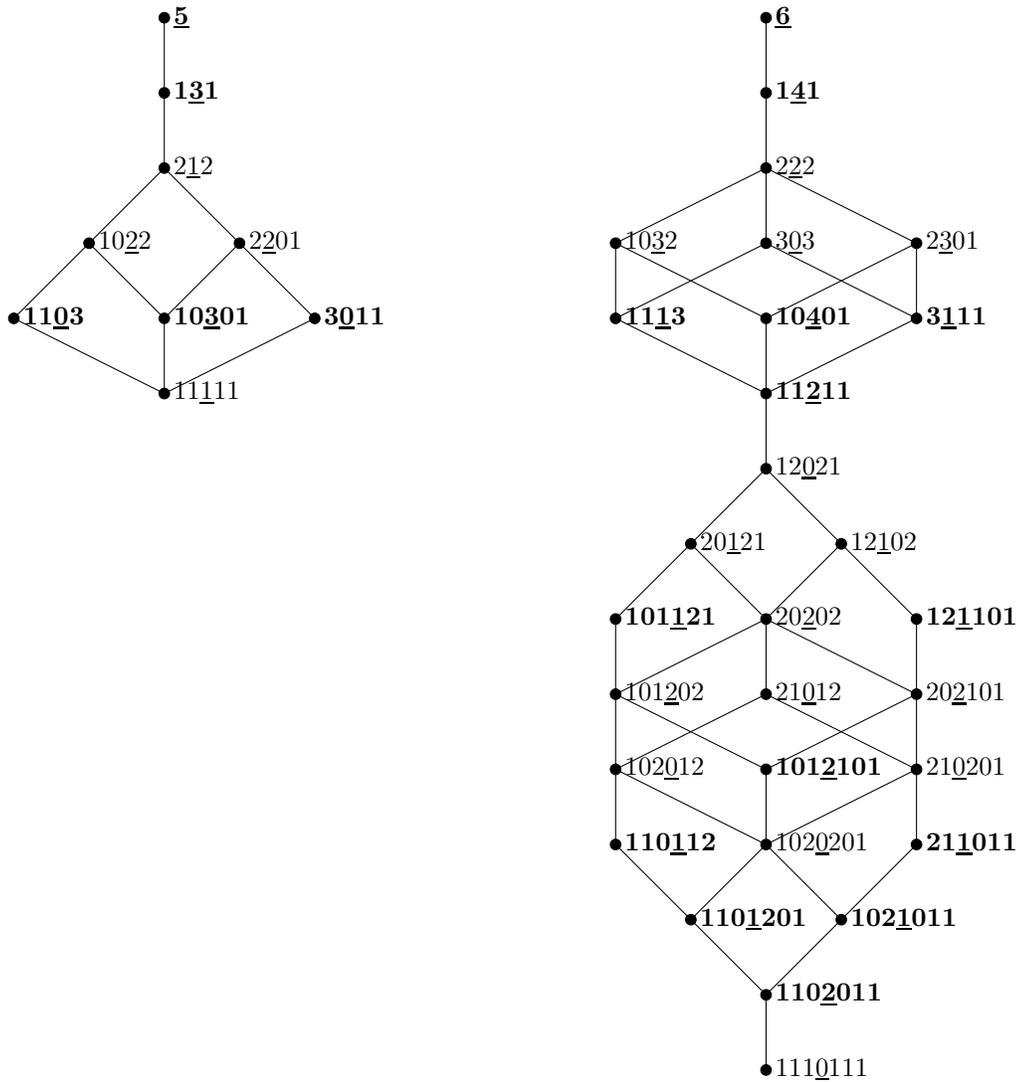
\begin{figure}
\centering
\begin{tikzpicture}[scale=1.0]

\filldraw[black](0,0) circle(2 pt) node[align=center, right] {\textbf{\underline{5}}};
\filldraw[black](0,-1) circle(2 pt) node[align=center, right] {\textbf{1\underline{3}1}};
\filldraw[black](0,-2) circle(2 pt) node[align=center, right] {2\underline{1}2};
\filldraw[black](-1,-3) circle(2 pt) node[align=center, right] {10\underline{2}2};
\filldraw[black](1,-3) circle(2 pt) node[align=center, right] {2\underline{2}01};
\filldraw[black](-2,-4) circle(2 pt) node[align=center, right] {\textbf{11\underline{0}3}};
\filldraw[black](0,-4) circle(2 pt) node[align=center, right] {\textbf{10\underline{3}01}};
\filldraw[black](2,-4) circle(2 pt) node[align=center, right] {\textbf{3\underline{0}11}};
\filldraw[black](0,-5) circle(2 pt) node[align=center, right] {11\underline{1}11};

\draw[black](0,0) -- (0,-2);
\draw[black](0,-2) -- (2,-4);
\draw[black](0,-2) -- (-2,-4);
\draw[black](-1,-3) -- (0,-4);
\draw[black](1,-3) -- (0,-4);
\draw[black](-2,-4) -- (0,-5);
\draw[black](0,-4) -- (0,-5);
\draw[black](2,-4) -- (0,-5);

\filldraw[black](8,0) circle(2 pt) node[align=center, right] {\textbf{\underline{6}}};
\filldraw[black](8,-1) circle(2 pt) node[align=center, right] {\textbf{1\underline{4}1}};
\filldraw[black](8,-2) circle(2 pt) node[align=center, right] {2\underline{2}2};
\filldraw[black](6,-3) circle(2 pt) node[align=center, right] {10\underline{3}2};
\filldraw[black](8,-3) circle(2 pt) node[align=center, right] {3\underline{0}3};
\filldraw[black](10,-3) circle(2 pt) node[align=center, right] {2\underline{3}01};
\filldraw[black](6,-4) circle(2 pt) node[align=center, right] {\textbf{11\underline{1}3}};
\filldraw[black](8,-4) circle(2 pt) node[align=center, right] {\textbf{10\underline{4}01}};
\filldraw[black](10,-4) circle(2 pt) node[align=center, right] {\textbf{3\underline{1}11}};
\filldraw[black](8,-5) circle(2 pt) node[align=center, right] {\textbf{11\underline{2}11}};
\filldraw[black](8,-6) circle(2 pt) node[align=center, right] {12\underline{0}21};
\filldraw[black](7,-7) circle(2 pt) node[align=center, right] {20\underline{1}21};
\filldraw[black](9,-7) circle(2 pt) node[align=center, right] {12\underline{1}02};
\filldraw[black](6,-8) circle(2 pt) node[align=center, right] {\textbf{101\underline{1}21}};
\filldraw[black](8,-8) circle(2 pt) node[align=center, right] {20\underline{2}02};
\filldraw[black](10,-8) circle(2 pt) node[align=center, right] {\textbf{12\underline{1}101}};
\filldraw[black](6,-9) circle(2 pt) node[align=center, right] {101\underline{2}02};
\filldraw[black](8,-9) circle(2 pt) node[align=center, right] {21\underline{0}12};
\filldraw[black](10,-9) circle(2 pt) node[align=center, right] {20\underline{2}101};
\filldraw[black](6,-10) circle(2 pt) node[align=center, right] {102\underline{0}12};
\filldraw[black](8,-10) circle(2 pt) node[align=center, right] {\textbf{101\underline{2}101}};
\filldraw[black](10,-10) circle(2 pt) node[align=center, right] {21\underline{0}201};
\filldraw[black](6,-11) circle(2 pt) node[align=center, right] {\textbf{110\underline{1}12}};
\filldraw[black](8,-11) circle(2 pt) node[align=center, right] {102\underline{0}201};
\filldraw[black](10,-11) circle(2 pt) node[align=center, right] {\textbf{21\underline{1}011}};
\filldraw[black](7,-12) circle(2 pt) node[align=center, right] {\textbf{110\underline{1}201}};
\filldraw[black](9,-12) circle(2 pt) node[align=center, right] {\textbf{102\underline{1}011}};
\filldraw[black](8,-13) circle(2 pt) node[align=center, right] {\textbf{110\underline{2}011}};
\filldraw[black](8,-14) circle(2 pt) node[align=center, right] {111\underline{0}111};

\draw[black](8,0) -- (8,-3);
\draw[black](8,-2) -- (6,-3);
\draw[black](8,-2) -- (10,-3);
\draw[black](6,-3) -- (6,-4);
\draw[black](6,-3) -- (8,-4);
\draw[black](8,-3) -- (6,-4);
\draw[black](8,-3) -- (10,-4);
\draw[black](10,-3) -- (8,-4);
\draw[black](10,-3) -- (10,-4);
\draw[black](6,-4) -- (8,-5);
\draw[black](8,-4) -- (8,-5);
\draw[black](10,-4) -- (8,-5);
\draw[black](8,-5) -- (8,-6);
\draw[black](8,-6) -- (6,-8);
\draw[black](8,-6) -- (10,-8);
\draw[black](7,-7) -- (8,-8);
\draw[black](9,-7) -- (8,-8);
\draw[black](6,-8) -- (6,-9);
\draw[black](8,-8) -- (6,-9);
\draw[black](8,-8) -- (8,-9);
\draw[black](8,-8) -- (10,-9);
\draw[black](10,-8) -- (10,-9);
\draw[black](6,-9) -- (6,-10);
\draw[black](6,-9) -- (8,-10);
\draw[black](8,-9) -- (6,-10);
\draw[black](8,-9) -- (10,-10);
\draw[black](10,-9) -- (8,-10);
\draw[black](10,-9) -- (10,-10);
\draw[black](6,-10) -- (6,-11);
\draw[black](6,-10) -- (8,-11);
\draw[black](8,-10) -- (8,-11);
\draw[black](10,-10) -- (8,-11);
\draw[black](10,-10) -- (10,-11);
\draw[black](6,-11) -- (8,-13);
\draw[black](10,-11) -- (8,-13);
\draw[black](8,-11) -- (7,-12);
\draw[black](8,-11) -- (9,-12);
\draw[black](8,-13) -- (8,-14);

\end{tikzpicture}
\caption{Configuration posets for chip-firing on a line with $n=5$ and $n=6$ chips.  5 is the smallest number of chips for which the poset does not form a distributive lattice, and 6 is the largest number of chips for which it does form a distributive lattice.  The number of chips at the origin is underlined, while the join irreducibles are in bold.}
  \label{figconfig56}
\setlength{\belowcaptionskip}{-10pt}
\end{figure}

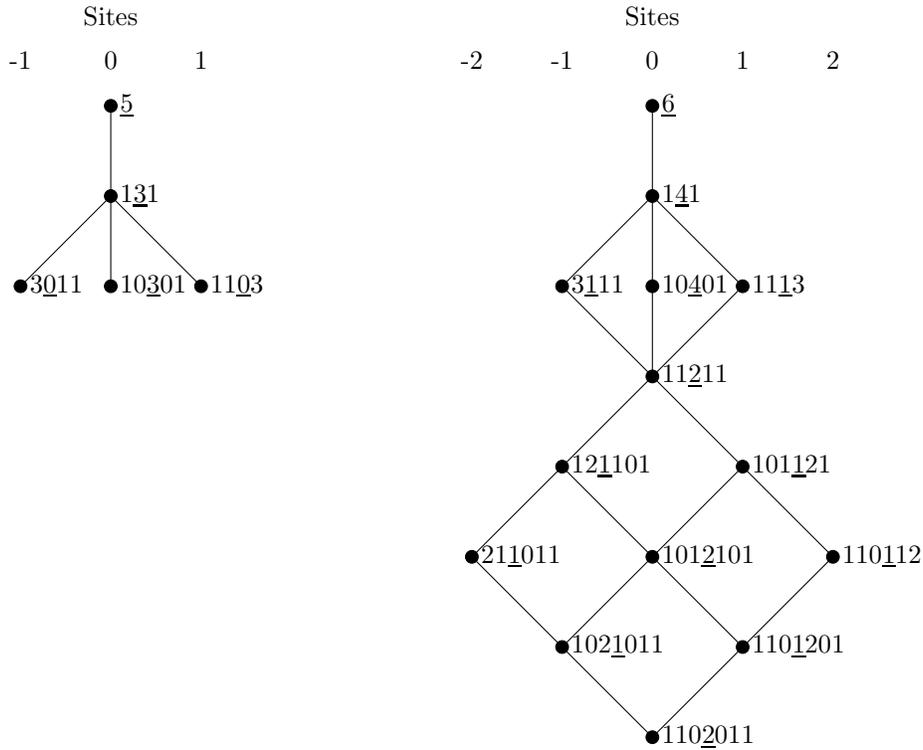
\begin{figure}
\centering
\begin{tikzpicture}[scale=1.2]

\node at (0,1) {Sites};
\node at (-1,0.5) {-1};
\node at (0,0.5) {0};
\node at (1,0.5) {1};

\filldraw[black](0,0) circle(2 pt) node[align=center, right] {\underline{5}};
\filldraw[black](0,-1) circle(2 pt) node[align=center, right] {1\underline{3}1};
\filldraw[black](0,-2) circle(2 pt) node[align=center, right] {10\underline{3}01};
\filldraw[black](1,-2) circle(2 pt) node[align=center, right] {11\underline{0}3};
\filldraw[black](-1,-2) circle(2 pt) node[align=center, right] {3\underline{0}11};

\draw[black](0,-1) -- (1,-2);
\draw[black](0,-1) -- (-1,-2);
\draw[black](0,0) -- (0,-2);

\node at (6,1) {Sites};
\node at (4,0.5) {-2};
\node at (5,0.5) {-1};
\node at (6,0.5) {0};
\node at (7,0.5) {1};
\node at (8,0.5) {2};

\filldraw[black](6,0) circle(2 pt) node[align=center, right] {\underline{6}};
\filldraw[black](6,-1) circle(2 pt) node[align=center, right] {1\underline{4}1};
\filldraw[black](6,-2) circle(2 pt) node[align=center, right] {10\underline{4}01};
\filldraw[black](6,-3) circle(2 pt) node[align=center, right] {11\underline{2}11};
\filldraw[black](6,-5) circle(2 pt) node[align=center, right] {101\underline{2}101};
\filldraw[black](6,-7) circle(2 pt) node[align=center, right] {110\underline{2}011};
\filldraw[black](7,-2) circle(2 pt) node[align=center, right] {11\underline{1}3};
\filldraw[black](5,-2) circle(2 pt) node[align=center, right] {3\underline{1}11};
\filldraw[black](7,-4) circle(2 pt) node[align=center, right] {101\underline{1}21};
\filldraw[black](5,-4) circle(2 pt) node[align=center, right] {12\underline{1}101};
\filldraw[black](7,-6) circle(2 pt) node[align=center, right] {110\underline{1}201};
\filldraw[black](5,-6) circle(2 pt) node[align=center, right] {102\underline{1}011};
\filldraw[black](8,-5) circle(2 pt) node[align=center, right] {110\underline{1}12};
\filldraw[black](4,-5) circle(2 pt) node[align=center, right] {21\underline{1}011};

\draw[black](6,-1) -- (7,-2);
\draw[black](6,-1) -- (5,-2);
\draw[black](6,-3) -- (7,-4);
\draw[black](6,-3) -- (5,-4);
\draw[black](6,-5) -- (7,-6);
\draw[black](6,-5) -- (5,-6);
\draw[black](7,-4) -- (8,-5);
\draw[black](5,-4) -- (4,-5);
\draw[black](7,-2) -- (6,-3);
\draw[black](5,-2) -- (6,-3);
\draw[black](7,-4) -- (6,-5);
\draw[black](5,-4) -- (6,-5);
\draw[black](7,-6) -- (6,-7);
\draw[black](5,-6) -- (6,-7);
\draw[black](8,-5) -- (7,-6);
\draw[black](4,-5) -- (5,-6);
\draw[black](6,0) -- (6,-3);

\end{tikzpicture}
\caption{The join-irreducibles of the configuration poset for chip-firing on a line with $n=5$ and $n=6$.  Each ``column'' in the hasse diagram corresponds to one site being ready to fire, with one join-irreducible corresponding to each firing move at that site.}
  \label{figjoin56}
\setlength{\belowcaptionskip}{-10pt}
\end{figure}

The theorem is illustrated in figures \ref{figconfig56} and \ref{figjoin56}.  We put $n$ chips at the origin of a 1 dimensional grid and then fire to completion.  Figure \ref{figconfig56} shows the configuration posets for $n=5$ and $n=6$, while Figure \ref{figjoin56} shows the join irreducibles of those posets, which are isomorphic to the move posets in those two cases.

Much work has been done in characterizing the configuration posets induced from chip-firing games.  Latapy and Phan show that this poset is a lattice for any chip-firing game \cite{LP00}.  In a finite chip-firing game, this implies global confluence, in which any terminating chip-firing game must have a unique final configuration.  Even in the infinite case, it implies that any two reachable configurations have a unique ``first'' configuration that can be reached from both of them.

Specifically, the lattice formed by the configurations of a chip-firing game must be upper locally distributive: the interval between an element and the meet of all its upper covers is a hypercube.  The class of lattices formed by configurations of chip-firing games furthermore includes the set of distributive lattices  \cite{MPV01} \cite{Magnien03}.

The move poset is particularly relevant to the notion of confluence.  As stated above, a terminating chip-firing process is globally confluent.  While this is implied by the lattice structure of the configuration poset, it is usually proven using a local confluence property: for any two configurations available from a given configuration after one move, there is a common configuration reachable from both resulting configurations in one additional move \cite{Klivans18}.  This property, combined with Newman's Lemma on abstract rewriting systems \cite{Newman42}, gives the global confluence property of chip-firing.

While chip-firing exhibits both local and global confluence, there are other similar systems that still have global confluence without the corresponding local property.  These include flow-firing \cite{FK19}, labeled chip-firing \cite{HMP16} \cite{KL20}, and root system chip-firing \cite{GHMP1} \cite{GHMP2}.

In labeled chip-firing, $n$ chips labeled from 1 to $n$ are placed at the origin of a 1D grid.  As in the traditional, unlabeled chip-firing process, a firing move consists of choosing 2 chips at the same site and sending one to the left and one to the right.  In this case, the chips are distinguishable, so we add the condition that for the two chips chosen, the smaller one is sent to the left and the larger one is sent to the right.  When $n$ is even, this process is still globally confluent, as the final positions of the $n$ chips must be in sorted order.  The structure of the move poset guarantees that a relatively small collection of locally confluent moves at the end of the process guarantee that all inversions between chips will be removed by the end of the process \cite{KL20}.

In section 2, we will prove that the elements of the move poset form the join-irreducibles of the configuration poset.  In section 3, we will show that for a particular class of chip-configurations on a 1-D grid, the lattice of configurations is in fact distributive.

\section{Join-irreducibles of configuration reachability lattice}

In this section, we prove Theorem \ref{join}: the join-irreducibles of the configuration poset form the move poset for any terminating chip-firing process.  Note that for a finite poset, the join-irreducibles correspond to elements that cover exactly one other element, so in a terminating chip-firing game, join-irreducibles correspond to configurations in which exactly one site is ready to fire.  For the results below, we assume a terminating chip-firing process beginning with a configuration $c$.

\begin{lemma}\label{onlymove}
For every firing move $k^j$, there exists a configuration, reachable from $c$, in which $k^j$ is the only move available.
\end{lemma}

\begin{proof}
We construct such a configuration as follows.  Consider a firing sequence $s$ that begins at $c$ and runs until completion.  Create a new firing sequence $s'$ that duplicates $s$ until $j-1$ firing moves have been performed at site $k$.  After reaching this move, proceed to perform all possible firing moves at sites other than $k$, until no such moves are available.  Once this is done, no sites other than $k$ may fire.  Furthermore, since move $k^j$ has not yet occurred, the configuration cannot be a final configuration in the process.  Therefore, site $k$ must be ready to fire, and the resulting firing move at that site is $k^j$.
\end{proof}

\begin{lemma}
The configuration from Lemma \ref{onlymove} is unique for each $k^j$.
\end{lemma}

\begin{proof}
Given a graph $G$ with initial chip configuration $c$, and a firing move $k^j$, we create an augmented graph $G'$ as follows.  Choose an $N$ greater than the number of chips in $c$.  Add vertices $v_1$ and $v_2$ to $G$, with $N$ edges from $k$ to $v_1$ and $Nj$ edges from $v_1$ to $v_2$.  Also add an additional $N(j-1)$ chips to site $k$ to create initial chip configuration $c'$.

Now, suppose that we are given a configuration $c_1$ meeting the conditions of Lemma 2.1, with corresponding firing sequence $s_1$ that goes from configuration $c$ to configuration $c_1$.  We then apply the firing moves in $s_1$ to the augmented graph $G'$ to produce a new configuration $c_1'$.  All firing moves that take place at sites in $G \backslash \{ k \}$ are exactly the same.  All moves at site $k$ send the same number of chips to adjacent sites in $G$, and they send an additional $N$ chips from $K$ to $v_1$.  $v_1$ cannot fire, because it has more incident edges than there are chips in the entire configuration, and $v_2$ never has any chips.  Thus, after $s_1$ is applied to the configuration on $G'$, all sites in $G \backslash \{ k \}$ have the same number of chips as they do in configuration $c_1$.  Site $k$ started off with $N(j-1)$ extra chips and has fired $j-1$ times, so it also has as many chips as it did in configuration $c_1$.  Furthermore, all firing moves performed were legal moves, and no site in $G'$ can fire after these moves have been performed, so the resulting configuration is a final configuration of the chip-firing process beginning with configuration $c'$ on $G'$.

Since final configurations of a finite chip-firing process are unique, $c_1'$ must be the unique final configuration of the chip-firing process beginning with configuration $c'$.  Furthermore, all sites in $G$ have the same number of chips in $c_1'$ as they do in $c_1$, so $c_1$ must be the only possible configuration meeting the conditions of Lemma 2.1.
\end{proof}

We now have a bijection between the move poset and the join-irreducibles of the configuration poset, which comes from each move having a unique configuration in which it is the only firing move available.  We will show that the posets themselves are isomorphic.

Define $c(k^j)$ to be the (unique) configuration in which $k^j$ is the only firing move available.

\begin{lemma}
If move $k^j$ must occur before move $k'^{j'}$, then $c(k'^{j'})$ is reachable from $c(k^j)$.
\end{lemma}

\begin{proof}
Since $k^j$ must occur before $k'^{j'}$, move $k'^{j'}$ must not happen before configuration $c(k^j)$ is reached.  Thus, if we start with $c(k^j)$ and perform all available firing moves other than $k'^{j'}$, we will eventually reach a state in which $k'^{j'}$ is the only move available.  By Lemma 2.2, this is the unique state $c(k'^{j'})$, so $c(k'^{j'})$ is reachable from $c(k^j)$.
\end{proof}

\begin{lemma}
If $k'^{j'}$ can occur before $k^j$, then $c(k'^{j'})$ is not reachable from $c(k^j)$.

\begin{proof}
Consider some firing sequence $s$ in which $k'^{j'}$ occurs before $k^j$.  We create another sequence by running sequence $s'$ up to $k^j$.  Instead of performing $k^j$, perform all other available moves except for $k^j$.  The resulting state is $c(k^j)$.  Thus, $k'^{j'}$ can occur before reaching configuration $c(k^j)$, and since a configuration uniquely corresponds to the moves that produced it, move $k'^{j'}$ must occur before reaching configuration $c(k^j)$.  Since $k'^{j'}$ cannot occur before reaching configuration $c(k'^{j'})$, it is not possible to reach $c(k'^{j'})$ from $c(k^j)$.
\end{proof}
\end{lemma}

\begin{theorem}\label{join}
The join-irreducibles of the configuration poset form the move poset.
\end{theorem}

\begin{proof}
Every join-irreducible of the configuration poset corresponds to the single move that can occur from that configuration.  By Lemmas 2.1 and 2.2, each firing move corresponds to a unique one of these join-irreducible configurations.  By Lemmas 2.3 and 2.4, $k^j\ge k'^{j'}$ in the move poset iff $c(k^j) \ge c(k'^{j'})$ in the configuration poset, so the join-irreducibles of the configuration poset form the move poset.
\end{proof}

\section{Distributive Lattices in Chip-Firing on the Line}

We now turn to a specific class of chip configurations: chip-firing on a 1-dimensional grid with $n$ chips.  This is the setup for the problem of labeled chip-firing, in which the chips are given labels from 1 to $n$ and fired in a manner that always sends smaller chips to the left and larger ones to the right.  The move posets for select values of $n$ are shown in figures \ref{fign10} and \ref{fign20}.  For any $m$, the number of firing moves at each site is the same for $n=2m$ and $n=2m+1$, so the move posets are similar in those two cases.  Note the diamonds that appear at the bottom of the diagrams for $n=10$ and $n=20$.  These are indeed known to exist for any even $n$ \cite{KL20}.

\begin{figure}
\centering
\begin{tikzpicture}[scale=0.4]

\node at (0,2) {Sites};
\node at (-4,1) {-4};
\node at (-3,1) {-3};
\node at (-2,1) {-2};
\node at (-1,1) {-1};
\node at (0,1) {0};
\node at (1,1) {1};
\node at (2,1) {2};
\node at (3,1) {3};
\node at (4,1) {4};

\node at (10,2) {Sites};
\node at (6,1) {-4};
\node at (7,1) {-3};
\node at (8,1) {-2};
\node at (9,1) {-1};
\node at (10,1) {0};
\node at (11,1) {1};
\node at (12,1) {2};
\node at (13,1) {3};
\node at (14,1) {4};

\draw[black, thick] (0,0) -- (0,-10);
\draw[black, thick] (1,-2) -- (1,-9);
\draw[black, thick] (-1,-2) -- (-1,-9);
\draw[black, thick] (2,-5) -- (2,-9);
\draw[black, thick] (-2,-5) -- (-2,-9);
\draw[black, thick] (0,-1) -- (1,-2);
\draw[black, thick] (0,-1) -- (-1,-2);
\draw[black, thick] (0,-3) -- (2,-5);
\draw[black, thick] (0,-4) -- (1,-5);
\draw[black, thick] (0,-3) -- (-2,-5);
\draw[black, thick] (0,-4) -- (-1,-5);
\draw[black, thick] (0,-6) -- (3,-9);
\draw[black, thick] (0,-7) -- (2,-9);
\draw[black, thick] (0,-8) -- (1,-9);
\draw[black, thick] (0,-6) -- (-3,-9);
\draw[black, thick] (0,-7) -- (-2,-9);
\draw[black, thick] (0,-8) -- (-1,-9);
\draw[black, thick] (0,-10) -- (4,-14);
\draw[black, thick] (-1,-11) -- (3,-15);
\draw[black, thick] (-2,-12) -- (2,-16);
\draw[black, thick] (-3,-13) -- (1,-17);
\draw[black, thick] (-4,-14) -- (0,-18);
\draw[black, thick] (0,-10) -- (-4,-14);
\draw[black, thick] (1,-11) -- (-3,-15);
\draw[black, thick] (2,-12) -- (-2,-16);
\draw[black, thick] (3,-13) -- (-1,-17);
\draw[black, thick] (4,-14) -- (0,-18);
\draw[black, thick] (1,-2) -- (0,-6);
\draw[black, thick] (1,-4) -- (0,-7);
\draw[black, thick] (1,-5) -- (0,-8);
\draw[black, thick] (1,-7) -- (0,-9);
\draw[black, thick] (1,-9) -- (0,-10);
\draw[black, thick] (-1,-2) -- (0,-6);
\draw[black, thick] (-1,-4) -- (0,-7);
\draw[black, thick] (-1,-5) -- (0,-8);
\draw[black, thick] (-1,-7) -- (0,-9);
\draw[black, thick] (-1,-9) -- (0,-10);
\draw[black, thick] (2,-8) -- (1,-9);
\draw[black, thick] (2,-9) -- (1,-11);
\draw[black, thick] (-2,-8) -- (-1,-9);
\draw[black, thick] (-2,-9) -- (-1,-11);
\draw[black, thick] (3,-9) -- (2,-12);
\draw[black, thick] (-3,-9) -- (-2,-12);

\filldraw[black] (0,0) circle (2pt);
\filldraw[black] (0,-1) circle (2pt);
\filldraw[black] (0,-2) circle (2pt);
\filldraw[black] (0,-3) circle (2pt);
\filldraw[black] (0,-4) circle (2pt);
\filldraw[black] (0,-5) circle (2pt);
\filldraw[black] (0,-6) circle (2pt);
\filldraw[black] (0,-7) circle (2pt);
\filldraw[black] (0,-8) circle (2pt);
\filldraw[black] (0,-9) circle (2pt);
\filldraw[black] (0,-10) circle (2pt);
\filldraw[black] (0,-12) circle (2pt);
\filldraw[black] (0,-14) circle (2pt);
\filldraw[black] (0,-16) circle (2pt);
\filldraw[black] (0,-18) circle (2pt);
\filldraw[black] (1,-2) circle (2pt);
\filldraw[black] (1,-4) circle (2pt);
\filldraw[black] (1,-5) circle (2pt);
\filldraw[black] (1,-7) circle (2pt);
\filldraw[black] (1,-8) circle (2pt);
\filldraw[black] (1,-9) circle (2pt);
\filldraw[black] (1,-11) circle (2pt);
\filldraw[black] (1,-13) circle (2pt);
\filldraw[black] (1,-15) circle (2pt);
\filldraw[black] (1,-17) circle (2pt);
\filldraw[black] (-1,-2) circle (2pt);
\filldraw[black] (-1,-4) circle (2pt);
\filldraw[black] (-1,-5) circle (2pt);
\filldraw[black] (-1,-7) circle (2pt);
\filldraw[black] (-1,-8) circle (2pt);
\filldraw[black] (-1,-9) circle (2pt);
\filldraw[black] (-1,-11) circle (2pt);
\filldraw[black] (-1,-13) circle (2pt);
\filldraw[black] (-1,-15) circle (2pt);
\filldraw[black] (-1,-17) circle (2pt);
\filldraw[black] (2,-5) circle (2pt);
\filldraw[black] (2,-8) circle (2pt);
\filldraw[black] (2,-9) circle (2pt);
\filldraw[black] (2,-12) circle (2pt);
\filldraw[black] (2,-14) circle (2pt);
\filldraw[black] (2,-16) circle (2pt);
\filldraw[black] (-2,-5) circle (2pt);
\filldraw[black] (-2,-8) circle (2pt);
\filldraw[black] (-2,-9) circle (2pt);
\filldraw[black] (-2,-12) circle (2pt);
\filldraw[black] (-2,-14) circle (2pt);
\filldraw[black] (-2,-16) circle (2pt);
\filldraw[black] (3,-9) circle (2pt);
\filldraw[black] (3,-13) circle (2pt);
\filldraw[black] (3,-15) circle (2pt);
\filldraw[black] (-3,-9) circle (2pt);
\filldraw[black] (-3,-13) circle (2pt);
\filldraw[black] (-3,-15) circle (2pt);
\filldraw[black] (4,-14) circle (2pt);
\filldraw[black] (-4,-14) circle (2pt);

\draw[black, thick] (10,0) -- (10,-18);
\draw[black, thick] (11,-2) -- (11,-17);
\draw[black, thick] (9,-2) -- (9,-17);
\draw[black, thick] (12,-5) -- (12,-16);
\draw[black, thick] (8,-5) -- (8,-16);
\draw[black, thick] (13,-9) -- (13,-15);
\draw[black, thick] (7,-9) -- (7,-15);
\draw[black, thick] (10,-1) -- (11,-2);
\draw[black, thick] (10,-1) -- (9,-2);
\draw[black, thick] (10,-3) -- (12,-5);
\draw[black, thick] (10,-4) -- (11,-5);
\draw[black, thick] (10,-3) -- (8,-5);
\draw[black, thick] (10,-4) -- (9,-5);
\draw[black, thick] (10,-6) -- (13,-9);
\draw[black, thick] (10,-7) -- (12,-9);
\draw[black, thick] (10,-8) -- (11,-9);
\draw[black, thick] (10,-6) -- (7,-9);
\draw[black, thick] (10,-7) -- (8,-9);
\draw[black, thick] (10,-8) -- (9,-9);
\draw[black, thick] (10,-10) -- (14,-14);
\draw[black, thick] (10,-12) -- (13,-15);
\draw[black, thick] (10,-14) -- (12,-16);
\draw[black, thick] (10,-16) -- (11,-17);
\draw[black, thick] (10,-10) -- (6,-14);
\draw[black, thick] (10,-12) -- (7,-15);
\draw[black, thick] (10,-14) -- (8,-16);
\draw[black, thick] (10,-16) -- (9,-17);
\draw[black, thick] (11,-2) -- (10,-7);
\draw[black, thick] (11,-4) -- (10,-8);
\draw[black, thick] (11,-5) -- (10,-9);
\draw[black, thick] (11,-8) -- (10,-10);
\draw[black, thick] (11,-9) -- (10,-12);
\draw[black, thick] (11,-11) -- (10,-14);
\draw[black, thick] (11,-13) -- (10,-16);
\draw[black, thick] (11,-15) -- (10,-18);
\draw[black, thick] (9,-2) -- (10,-7);
\draw[black, thick] (9,-4) -- (10,-8);
\draw[black, thick] (9,-5) -- (10,-9);
\draw[black, thick] (9,-8) -- (10,-10);
\draw[black, thick] (9,-9) -- (10,-12);
\draw[black, thick] (9,-11) -- (10,-14);
\draw[black, thick] (9,-13) -- (10,-16);
\draw[black, thick] (9,-15) -- (10,-18);
\draw[black, thick] (12,-5) -- (11,-9);
\draw[black, thick] (12,-8) -- (11,-11);
\draw[black, thick] (12,-9) -- (11,-13);
\draw[black, thick] (12,-12) -- (11,-15);
\draw[black, thick] (12,-14) -- (11,-17);
\draw[black, thick] (8,-5) -- (9,-9);
\draw[black, thick] (8,-8) -- (9,-11);
\draw[black, thick] (8,-9) -- (9,-13);
\draw[black, thick] (8,-12) -- (9,-15);
\draw[black, thick] (8,-14) -- (9,-17);
\draw[black, thick] (13,-9) -- (12,-14);
\draw[black, thick] (13,-13) -- (12,-16);
\draw[black, thick] (7,-9) -- (8,-14);
\draw[black, thick] (7,-13) -- (8,-16);

\filldraw[black] (10,0) circle (2pt);
\filldraw[black] (10,-1) circle (2pt);
\filldraw[black] (10,-2) circle (2pt);
\filldraw[black] (10,-3) circle (2pt);
\filldraw[black] (10,-4) circle (2pt);
\filldraw[black] (10,-5) circle (2pt);
\filldraw[black] (10,-6) circle (2pt);
\filldraw[black] (10,-7) circle (2pt);
\filldraw[black] (10,-8) circle (2pt);
\filldraw[black] (10,-9) circle (2pt);
\filldraw[black] (10,-10) circle (2pt);
\filldraw[black] (10,-12) circle (2pt);
\filldraw[black] (10,-14) circle (2pt);
\filldraw[black] (10,-16) circle (2pt);
\filldraw[black] (10,-18) circle (2pt);
\filldraw[black] (11,-2) circle (2pt);
\filldraw[black] (11,-4) circle (2pt);
\filldraw[black] (11,-5) circle (2pt);
\filldraw[black] (11,-7) circle (2pt);
\filldraw[black] (11,-8) circle (2pt);
\filldraw[black] (11,-9) circle (2pt);
\filldraw[black] (11,-11) circle (2pt);
\filldraw[black] (11,-13) circle (2pt);
\filldraw[black] (11,-15) circle (2pt);
\filldraw[black] (11,-17) circle (2pt);
\filldraw[black] (9,-2) circle (2pt);
\filldraw[black] (9,-4) circle (2pt);
\filldraw[black] (9,-5) circle (2pt);
\filldraw[black] (9,-7) circle (2pt);
\filldraw[black] (9,-8) circle (2pt);
\filldraw[black] (9,-9) circle (2pt);
\filldraw[black] (9,-11) circle (2pt);
\filldraw[black] (9,-13) circle (2pt);
\filldraw[black] (9,-15) circle (2pt);
\filldraw[black] (9,-17) circle (2pt);
\filldraw[black] (12,-5) circle (2pt);
\filldraw[black] (12,-8) circle (2pt);
\filldraw[black] (12,-9) circle (2pt);
\filldraw[black] (12,-12) circle (2pt);
\filldraw[black] (12,-14) circle (2pt);
\filldraw[black] (12,-16) circle (2pt);
\filldraw[black] (8,-5) circle (2pt);
\filldraw[black] (8,-8) circle (2pt);
\filldraw[black] (8,-9) circle (2pt);
\filldraw[black] (8,-12) circle (2pt);
\filldraw[black] (8,-14) circle (2pt);
\filldraw[black] (8,-16) circle (2pt);
\filldraw[black] (13,-9) circle (2pt);
\filldraw[black] (13,-13) circle (2pt);
\filldraw[black] (13,-15) circle (2pt);
\filldraw[black] (7,-9) circle (2pt);
\filldraw[black] (7,-13) circle (2pt);
\filldraw[black] (7,-15) circle (2pt);
\filldraw[black] (14,-14) circle (2pt);
\filldraw[black] (6,-14) circle (2pt);

\end{tikzpicture}
\caption{Firing order poset for $n=10$ (left) and $n=11$ (right).}
  \label{fign10}
\setlength{\belowcaptionskip}{-10pt}
\end{figure}
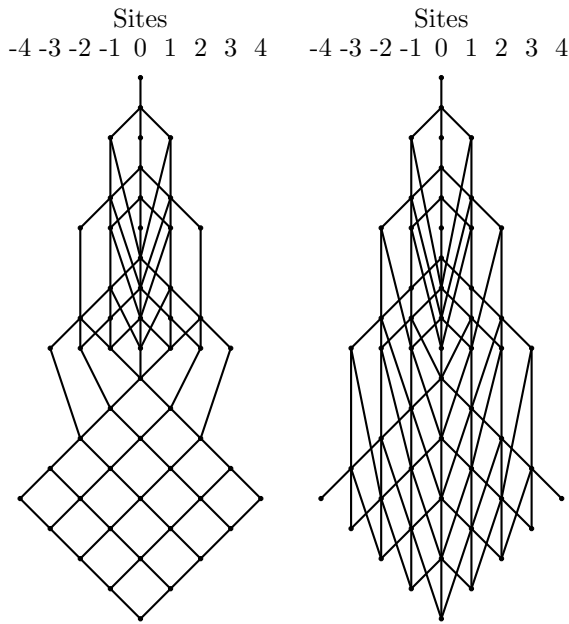

\begin{figure}
  \begin{center}
    \input{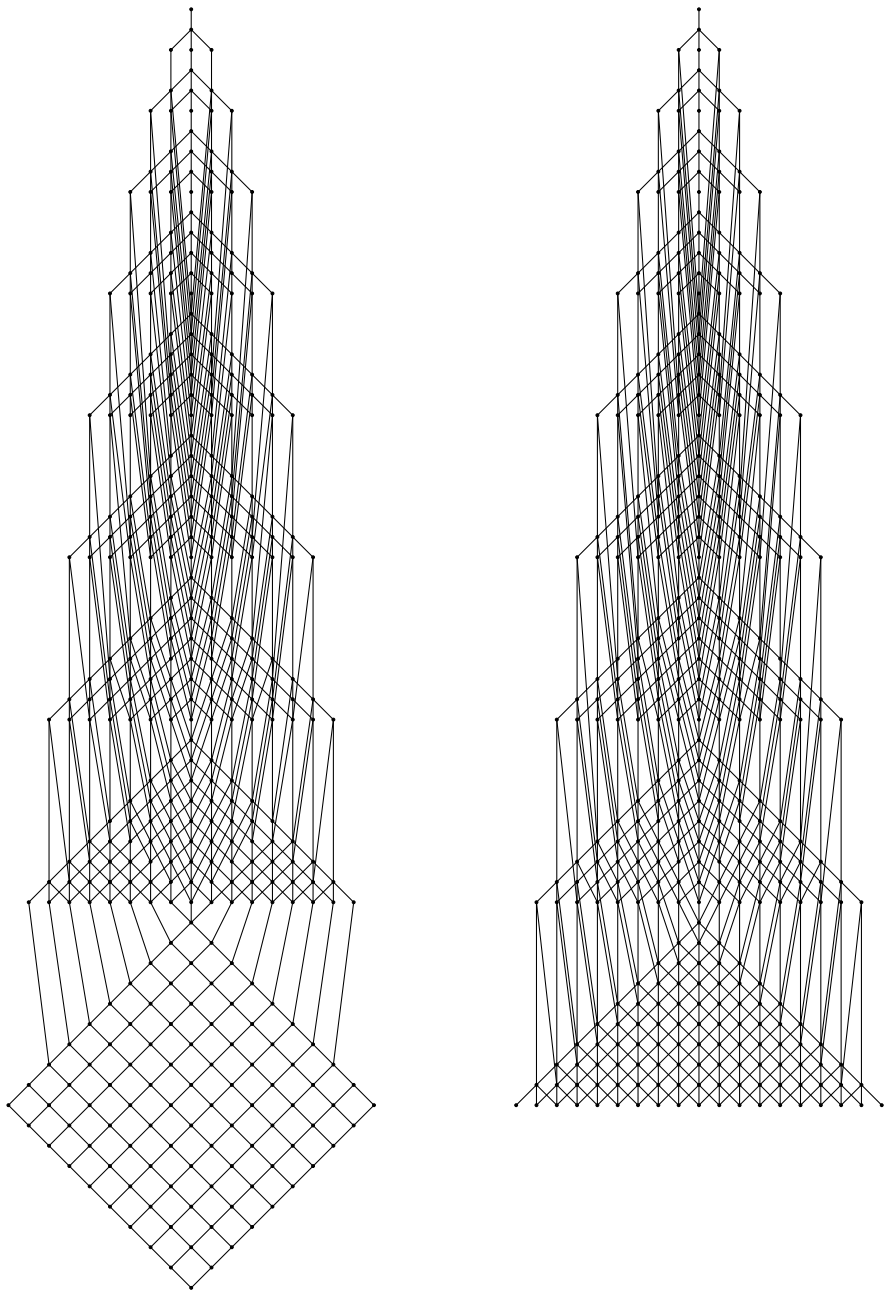}
\end{center}
\caption{Firing order poset for $n=20$ (left) and $n=21$ (right).}
  \label{fign20}
\end{figure}

In fact, if we only consider the diamond at the bottom of the hasse diagrams in the even case, the resulting figure forms the hasse diagram for the move poset for a specific class of chip configurations.  If we have $n=2m$ chips, then the diamond represents the configuration poset beginning with a configuration with 1 chip each at sites $-(m-1)$ through $-1$ and sites 1 through $m-1$, and 2 chips at site 0.  We call this configuration $c$.

The nice structure of the move poset suggests that there may be more structure to the poset configurations.  Jim Propp conjectured that the configuration poset may actually form a distributive lattice, and we in fact show that this is the case.

\begin{figure}
\centering
\begin{tikzpicture}[scale=0.4]

\node at (0,2) {Sites};
\node at (-4,1) {-4};
\node at (-3,1) {-3};
\node at (-2,1) {-2};
\node at (-1,1) {-1};
\node at (0,1) {0};
\node at (1,1) {1};
\node at (2,1) {2};
\node at (3,1) {3};
\node at (4,1) {4};


\draw[black, thick] (0,0) -- (4,-4);
\draw[black, thick] (-1,-1) -- (3,-5);
\draw[black, thick] (-2,-2) -- (2,-6);
\draw[black, thick] (-3,-3) -- (1,-7);
\draw[black, thick] (-4,-4) -- (0,-8);
\draw[black, thick] (0,0) -- (-4,-4);
\draw[black, thick] (1,-1) -- (-3,-5);
\draw[black, thick] (2,-2) -- (-2,-6);
\draw[black, thick] (3,-3) -- (-1,-7);
\draw[black, thick] (4,-4) -- (0,-8);

\filldraw[black] (0,0) circle (2pt);
\filldraw[black] (0,-2) circle (2pt);
\filldraw[black] (0,-4) circle (2pt);
\filldraw[black] (0,-6) circle (2pt);
\filldraw[black] (0,-8) circle (2pt);
\filldraw[black] (1,-1) circle (2pt);
\filldraw[black] (1,-3) circle (2pt);
\filldraw[black] (1,-5) circle (2pt);
\filldraw[black] (1,-7) circle (2pt);
\filldraw[black] (-1,-1) circle (2pt);
\filldraw[black] (-1,-3) circle (2pt);
\filldraw[black] (-1,-5) circle (2pt);
\filldraw[black] (-1,-7) circle (2pt);
\filldraw[black] (2,-2) circle (2pt);
\filldraw[black] (2,-4) circle (2pt);
\filldraw[black] (2,-6) circle (2pt);
\filldraw[black] (-2,-2) circle (2pt);
\filldraw[black] (-2,-4) circle (2pt);
\filldraw[black] (-2,-6) circle (2pt);
\filldraw[black] (3,-3) circle (2pt);
\filldraw[black] (3,-5) circle (2pt);
\filldraw[black] (-3,-3) circle (2pt);
\filldraw[black] (-3,-5) circle (2pt);
\filldraw[black] (4,-4) circle (2pt);
\filldraw[black] (-4,-4) circle (2pt);

\end{tikzpicture}
\caption{Firing order poset for the endgame with $n=10$.}
  \label{fign10end}
\setlength{\belowcaptionskip}{-10pt}
\end{figure}
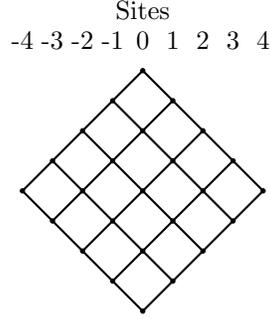

\begin{figure}
\centering
\begin{tikzpicture}[scale=0.27]

\filldraw[black](0,-45) circle(2 pt);
\filldraw[black](0,-47) circle(2 pt);
\filldraw[black](0,-49) circle(2 pt);
\filldraw[black](0,-51) circle(2 pt);
\filldraw[black](0,-53) circle(2 pt);
\filldraw[black](0,-55) circle(2 pt);
\filldraw[black](0,-57) circle(2 pt);
\filldraw[black](0,-59) circle(2 pt);
\filldraw[black](0,-61) circle(2 pt);
\filldraw[black](0,-63) circle(2 pt);
\filldraw[black](1,-46) circle(2 pt);
\filldraw[black](-1,-46) circle(2 pt);
\filldraw[black](1,-48) circle(2 pt);
\filldraw[black](-1,-48) circle(2 pt);
\filldraw[black](1,-50) circle(2 pt);
\filldraw[black](-1,-50) circle(2 pt);
\filldraw[black](1,-52) circle(2 pt);
\filldraw[black](-1,-52) circle(2 pt);
\filldraw[black](1,-54) circle(2 pt);
\filldraw[black](-1,-54) circle(2 pt);
\filldraw[black](1,-56) circle(2 pt);
\filldraw[black](-1,-56) circle(2 pt);
\filldraw[black](1,-58) circle(2 pt);
\filldraw[black](-1,-58) circle(2 pt);
\filldraw[black](1,-60) circle(2 pt);
\filldraw[black](-1,-60) circle(2 pt);
\filldraw[black](1,-62) circle(2 pt);
\filldraw[black](-1,-62) circle(2 pt);
\filldraw[black](2,-47) circle(2 pt);
\filldraw[black](-2,-47) circle(2 pt);
\filldraw[black](2,-49) circle(2 pt);
\filldraw[black](-2,-49) circle(2 pt);
\filldraw[black](2,-51) circle(2 pt);
\filldraw[black](-2,-51) circle(2 pt);
\filldraw[black](2,-53) circle(2 pt);
\filldraw[black](-2,-53) circle(2 pt);
\filldraw[black](2,-55) circle(2 pt);
\filldraw[black](-2,-55) circle(2 pt);
\filldraw[black](2,-57) circle(2 pt);
\filldraw[black](-2,-57) circle(2 pt);
\filldraw[black](2,-59) circle(2 pt);
\filldraw[black](-2,-59) circle(2 pt);
\filldraw[black](2,-61) circle(2 pt);
\filldraw[black](-2,-61) circle(2 pt);
\filldraw[black](3,-48) circle(2 pt);
\filldraw[black](-3,-48) circle(2 pt);
\filldraw[black](3,-50) circle(2 pt);
\filldraw[black](-3,-50) circle(2 pt);
\filldraw[black](3,-52) circle(2 pt);
\filldraw[black](-3,-52) circle(2 pt);
\filldraw[black](3,-54) circle(2 pt);
\filldraw[black](-3,-54) circle(2 pt);
\filldraw[black](3,-56) circle(2 pt);
\filldraw[black](-3,-56) circle(2 pt);
\filldraw[black](3,-58) circle(2 pt);
\filldraw[black](-3,-58) circle(2 pt);
\filldraw[black](3,-60) circle(2 pt);
\filldraw[black](-3,-60) circle(2 pt);
\filldraw[black](4,-49) circle(2 pt);
\filldraw[black](-4,-49) circle(2 pt);
\filldraw[black](4,-51) circle(2 pt);
\filldraw[black](-4,-51) circle(2 pt);
\filldraw[black](4,-53) circle(2 pt);
\filldraw[black](-4,-53) circle(2 pt);
\filldraw[black](4,-55) circle(2 pt);
\filldraw[black](-4,-55) circle(2 pt);
\filldraw[black](4,-57) circle(2 pt);
\filldraw[black](-4,-57) circle(2 pt);
\filldraw[black](4,-59) circle(2 pt);
\filldraw[black](-4,-59) circle(2 pt);
\filldraw[black](5,-50) circle(2 pt);
\filldraw[black](-5,-50) circle(2 pt);
\filldraw[black](5,-52) circle(2 pt);
\filldraw[black](-5,-52) circle(2 pt);
\filldraw[black](5,-54) circle(2 pt);
\filldraw[black](-5,-54) circle(2 pt);
\filldraw[black](5,-56) circle(2 pt);
\filldraw[black](-5,-56) circle(2 pt);
\filldraw[black](5,-58) circle(2 pt);
\filldraw[black](-5,-58) circle(2 pt);
\filldraw[black](6,-51) circle(2 pt);
\filldraw[black](-6,-51) circle(2 pt);
\filldraw[black](6,-53) circle(2 pt);
\filldraw[black](-6,-53) circle(2 pt);
\filldraw[black](6,-55) circle(2 pt);
\filldraw[black](-6,-55) circle(2 pt);
\filldraw[black](6,-57) circle(2 pt);
\filldraw[black](-6,-57) circle(2 pt);
\filldraw[black](7,-52) circle(2 pt);
\filldraw[black](-7,-52) circle(2 pt);
\filldraw[black](7,-54) circle(2 pt);
\filldraw[black](-7,-54) circle(2 pt);
\filldraw[black](7,-56) circle(2 pt);
\filldraw[black](-7,-56) circle(2 pt);
\filldraw[black](8,-53) circle(2 pt);
\filldraw[black](-8,-53) circle(2 pt);
\filldraw[black](8,-55) circle(2 pt);
\filldraw[black](-8,-55) circle(2 pt);
\filldraw[black](9,-54) circle(2 pt);
\filldraw[black](-9,-54) circle(2 pt);

\draw[black](0,-45) -- (1,-46);
\draw[black](0,-45) -- (-1,-46);
\draw[black](0,-47) -- (1,-48);
\draw[black](0,-47) -- (-1,-48);
\draw[black](0,-49) -- (1,-50);
\draw[black](0,-49) -- (-1,-50);
\draw[black](0,-51) -- (1,-52);
\draw[black](0,-51) -- (-1,-52);
\draw[black](0,-53) -- (1,-54);
\draw[black](0,-53) -- (-1,-54);
\draw[black](0,-55) -- (1,-56);
\draw[black](0,-55) -- (-1,-56);
\draw[black](0,-57) -- (1,-58);
\draw[black](0,-57) -- (-1,-58);
\draw[black](0,-59) -- (1,-60);
\draw[black](0,-59) -- (-1,-60);
\draw[black](0,-61) -- (1,-62);
\draw[black](0,-61) -- (-1,-62);
\draw[black](1,-46) -- (2,-47);
\draw[black](-1,-46) -- (-2,-47);
\draw[black](1,-48) -- (2,-49);
\draw[black](-1,-48) -- (-2,-49);
\draw[black](1,-50) -- (2,-51);
\draw[black](-1,-50) -- (-2,-51);
\draw[black](1,-52) -- (2,-53);
\draw[black](-1,-52) -- (-2,-53);
\draw[black](1,-54) -- (2,-55);
\draw[black](-1,-54) -- (-2,-55);
\draw[black](1,-56) -- (2,-57);
\draw[black](-1,-56) -- (-2,-57);
\draw[black](1,-58) -- (2,-59);
\draw[black](-1,-58) -- (-2,-59);
\draw[black](1,-60) -- (2,-61);
\draw[black](-1,-60) -- (-2,-61);
\draw[black](2,-47) -- (3,-48);
\draw[black](-2,-47) -- (-3,-48);
\draw[black](2,-49) -- (3,-50);
\draw[black](-2,-49) -- (-3,-50);
\draw[black](2,-51) -- (3,-52);
\draw[black](-2,-51) -- (-3,-52);
\draw[black](2,-53) -- (3,-54);
\draw[black](-2,-53) -- (-3,-54);
\draw[black](2,-55) -- (3,-56);
\draw[black](-2,-55) -- (-3,-56);
\draw[black](2,-57) -- (3,-58);
\draw[black](-2,-57) -- (-3,-58);
\draw[black](2,-59) -- (3,-60);
\draw[black](-2,-59) -- (-3,-60);
\draw[black](3,-48) -- (4,-49);
\draw[black](-3,-48) -- (-4,-49);
\draw[black](3,-50) -- (4,-51);
\draw[black](-3,-50) -- (-4,-51);
\draw[black](3,-52) -- (4,-53);
\draw[black](-3,-52) -- (-4,-53);
\draw[black](3,-54) -- (4,-55);
\draw[black](-3,-54) -- (-4,-55);
\draw[black](3,-56) -- (4,-57);
\draw[black](-3,-56) -- (-4,-57);
\draw[black](3,-58) -- (4,-59);
\draw[black](-3,-58) -- (-4,-59);
\draw[black](4,-49) -- (5,-50);
\draw[black](-4,-49) -- (-5,-50);
\draw[black](4,-51) -- (5,-52);
\draw[black](-4,-51) -- (-5,-52);
\draw[black](4,-53) -- (5,-54);
\draw[black](-4,-53) -- (-5,-54);
\draw[black](4,-55) -- (5,-56);
\draw[black](-4,-55) -- (-5,-56);
\draw[black](4,-57) -- (5,-58);
\draw[black](-4,-57) -- (-5,-58);
\draw[black](5,-50) -- (6,-51);
\draw[black](-5,-50) -- (-6,-51);
\draw[black](5,-52) -- (6,-53);
\draw[black](-5,-52) -- (-6,-53);
\draw[black](5,-54) -- (6,-55);
\draw[black](-5,-54) -- (-6,-55);
\draw[black](5,-56) -- (6,-57);
\draw[black](-5,-56) -- (-6,-57);
\draw[black](6,-51) -- (7,-52);
\draw[black](-6,-51) -- (-7,-52);
\draw[black](6,-53) -- (7,-54);
\draw[black](-6,-53) -- (-7,-54);
\draw[black](6,-55) -- (7,-56);
\draw[black](-6,-55) -- (-7,-56);
\draw[black](7,-52) -- (8,-53);
\draw[black](-7,-52) -- (-8,-53);
\draw[black](7,-54) -- (8,-55);
\draw[black](-7,-54) -- (-8,-55);
\draw[black](8,-53) -- (9,-54);
\draw[black](-8,-53) -- (-9,-54);
\draw[black](1,-46) -- (0,-47);
\draw[black](-1,-46) -- (0,-47);
\draw[black](1,-48) -- (0,-49);
\draw[black](-1,-48) -- (0,-49);
\draw[black](1,-50) -- (0,-51);
\draw[black](-1,-50) -- (0,-51);
\draw[black](1,-52) -- (0,-53);
\draw[black](-1,-52) -- (0,-53);
\draw[black](1,-54) -- (0,-55);
\draw[black](-1,-54) -- (0,-55);
\draw[black](1,-56) -- (0,-57);
\draw[black](-1,-56) -- (0,-57);
\draw[black](1,-58) -- (0,-59);
\draw[black](-1,-58) -- (0,-59);
\draw[black](1,-60) -- (0,-61);
\draw[black](-1,-60) -- (0,-61);
\draw[black](1,-62) -- (0,-63);
\draw[black](-1,-62) -- (0,-63);
\draw[black](2,-47) -- (1,-48);
\draw[black](-2,-47) -- (-1,-48);
\draw[black](2,-49) -- (1,-50);
\draw[black](-2,-49) -- (-1,-50);
\draw[black](2,-51) -- (1,-52);
\draw[black](-2,-51) -- (-1,-52);
\draw[black](2,-53) -- (1,-54);
\draw[black](-2,-53) -- (-1,-54);
\draw[black](2,-55) -- (1,-56);
\draw[black](-2,-55) -- (-1,-56);
\draw[black](2,-57) -- (1,-58);
\draw[black](-2,-57) -- (-1,-58);
\draw[black](2,-59) -- (1,-60);
\draw[black](-2,-59) -- (-1,-60);
\draw[black](2,-61) -- (1,-62);
\draw[black](-2,-61) -- (-1,-62);
\draw[black](3,-48) -- (2,-49);
\draw[black](-3,-48) -- (-2,-49);
\draw[black](3,-50) -- (2,-51);
\draw[black](-3,-50) -- (-2,-51);
\draw[black](3,-52) -- (2,-53);
\draw[black](-3,-52) -- (-2,-53);
\draw[black](3,-54) -- (2,-55);
\draw[black](-3,-54) -- (-2,-55);
\draw[black](3,-56) -- (2,-57);
\draw[black](-3,-56) -- (-2,-57);
\draw[black](3,-58) -- (2,-59);
\draw[black](-3,-58) -- (-2,-59);
\draw[black](3,-60) -- (2,-61);
\draw[black](-3,-60) -- (-2,-61);
\draw[black](4,-49) -- (3,-50);
\draw[black](-4,-49) -- (-3,-50);
\draw[black](4,-51) -- (3,-52);
\draw[black](-4,-51) -- (-3,-52);
\draw[black](4,-53) -- (3,-54);
\draw[black](-4,-53) -- (-3,-54);
\draw[black](4,-55) -- (3,-56);
\draw[black](-4,-55) -- (-3,-56);
\draw[black](4,-57) -- (3,-58);
\draw[black](-4,-57) -- (-3,-58);
\draw[black](4,-59) -- (3,-60);
\draw[black](-4,-59) -- (-3,-60);
\draw[black](5,-50) -- (4,-51);
\draw[black](-5,-50) -- (-4,-51);
\draw[black](5,-52) -- (4,-53);
\draw[black](-5,-52) -- (-4,-53);
\draw[black](5,-54) -- (4,-55);
\draw[black](-5,-54) -- (-4,-55);
\draw[black](5,-56) -- (4,-57);
\draw[black](-5,-56) -- (-4,-57);
\draw[black](5,-58) -- (4,-59);
\draw[black](-5,-58) -- (-4,-59);
\draw[black](6,-51) -- (5,-52);
\draw[black](-6,-51) -- (-5,-52);
\draw[black](6,-53) -- (5,-54);
\draw[black](-6,-53) -- (-5,-54);
\draw[black](6,-55) -- (5,-56);
\draw[black](-6,-55) -- (-5,-56);
\draw[black](6,-57) -- (5,-58);
\draw[black](-6,-57) -- (-5,-58);
\draw[black](7,-52) -- (6,-53);
\draw[black](-7,-52) -- (-6,-53);
\draw[black](7,-54) -- (6,-55);
\draw[black](-7,-54) -- (-6,-55);
\draw[black](7,-56) -- (6,-57);
\draw[black](-7,-56) -- (-6,-57);
\draw[black](8,-53) -- (7,-54);
\draw[black](-8,-53) -- (-7,-54);
\draw[black](8,-55) -- (7,-56);
\draw[black](-8,-55) -- (-7,-56);
\draw[black](9,-54) -- (8,-55);
\draw[black](-9,-54) -- (-8,-55);

\end{tikzpicture}
\caption{Firing order poset for the endgame with $n=20$.}
  \label{fign20end}
\setlength{\belowcaptionskip}{-10pt}
\end{figure}
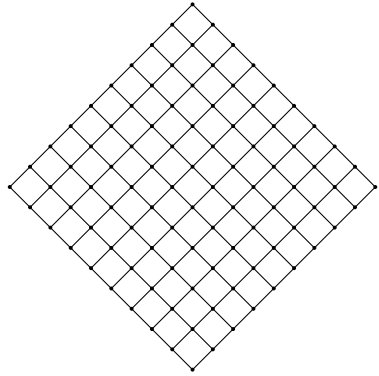

We will show that the lattice of configurations in the ``endgame'' of chip-firing on a line is distributive, and then provide a more direct proof that the join-irreducibles of the poset are elements of the firing-move poset.

We will provide a simpler version of the proof of Theorem \ref{join} for this special case.  Define $mv(c_1)$ to be the set of moves needed to get to state $c_1$ from state $c$.  It is known that this is uniquely determined for a state $c_1$, even if the order of the moves is not.

\begin{lemma}
If $mv(c_1) \subseteq mv(c_2)$, then it is possible to get from $c_1$ to $c_2$ through a sequence of firing moves.
\end{lemma}

\begin{proof}
We proceed by induction on $|mv(c_2) \backslash mv(c_1)|$.  If the cardinality is 0, then $c_1=c_2$, so it is possible to get from $c_1$ to $c_2$ in 0 firing moves.  Now, suppose that it is possible to find such a sequence for all $c_1,c_2$ such that $|mv(c_2) \backslash mv(c_1)|\le n$.  We then consider some $c_1,c_2$ such that $|mv(c_2) \backslash mv(c_1)|= n+1$.  Consider $mv(c_2) \backslash mv(c_1)$ as a subposet of the move poset, and then take a maximal element $k^j$ of the subposet.  Because $k^j \in mv(c_2)$, its covers in the diamond must also be in $mv(c_2)$, and because it is a maximal element of $mv(c_2) \backslash mv(c_1)$, its covers must also be in $mv(c_1)$, so it is possible to perform the firing move $k^j$ from configuration $c_1$.

This yields another configuration $c_1'$ such that $mv(c_1') \subseteq mv(c_2)$ and $|mv(c_2) \backslash mv(c_1')|= n$.  By the inductive hypothesis, it is possible to reach $c_2$ from $c_1'$ through a sequence of firing moves, so it must also be possible to reach $c_2$ from $c_1$.
\end{proof}

\begin{cor}
Any upward closed collection of moves $S$ is equal to $mv(c_1)$ for some configuration $c_1$.
\end{cor}

\begin{proof}
We have that $mv(c)=\emptyset$.  By the proof of Lemma 1.1, we can apply all of the moves of $S$ in some order to get from $c$ to some new state.
\end{proof}

\begin{lemma}
The configuration poset forms a lattice where $c_1 \vee c_2$ is the unique state $c_3$ such that $mv(c_3)=mv(c_1) \cap mv(c_2)$, and $c_1 \wedge c_2$ is the unique state $c_4$ such that $mv(c_4)=mv(c_1) \cup mv(c_2)$.
\end{lemma}

\begin{proof}
Given configurations $c_1$ and $c_2$, consider the set of moves $m_3=mv(c_1) \cap mv(c_2)$.  By Corollary 1.2, this corresponds to a valid configuration reachable from $c$, which we call $c_3$.  By Lemma 1.1, it is possible to reach $c_1$ or $c_2$ from $c_3$ by a sequence of firing moves, so this configuration is an upper bound for $c_1$ and $c_2$.  Now, any upper bound $c_3'$ must satisfy $mv(c_3') \subseteq mv(c_1) \cap mv(c_2)$ because it is not possible to get from one configuration to another configuration in which fewer moves have occurred at a given site.  As a result, any other upper bound $c_3'$ must satisfy $c_3' \ge c_3$, again by Lemma 1.1, so $c_3=c_1 \vee c_2$.

Now, given $c_1$ and $c_2$, consider the set of moves $m_4=mv(c_1) \cup mv(c_2)$.  By Corollary 1.2, this corresponds to a valid configuration reachable from $c$, which we call $c_4$.  By Lemma 1.1, it is possible to reach $c_4$ from $c_1$ or $c_2$ by a sequence of firing moves, so this configuration is a lower bound for $c_1$ and $c_2$.  Now, any lower bound $c_4'$ must satisfy $mv(c_1) \cup mv(c_2) \subseteq mv(c_4')$ because it is not possible to get from one configuration to another configuration in which fewer moves have occurred at a given site.  As a result, any other lower bound $c_4'$ must satisfy $c_4' \le c_4$, again by Lemma 1.1, so $c_4=c_1 \wedge c_2$.
\end{proof}

\begin{lemma}
The configuration poset forms a distributive lattice.
\end{lemma}

\begin{proof}
Since the meets and joins of two configurations correspond to the intersection and union of the moves needed to obtain them, the distributivity relations follow from the distributivity of set union and intersection.
\end{proof}

\begin{theorem}
The join-irreducibles of the configuration poset correspond to elements of the move poset.
\end{theorem}

\begin{proof}
We show that the order ideals of the move poset form the configuration poset.  Given an order ideal of the move poset, its complement is an upward-closed subposet $S$ of the move poset, which must satisfy $S=mv(c_1)$ for some configuration $c_1$ by Corollary 1.2.  If a configuration $c_2$ can be reached from another configuration $c_1$, then $mv(c_1) \subseteq mv(c_2)$, and the converse is true by Lemma 1.1.  Thus, the order ideals of the move poset form a subposet of the configuration poset.  Since every configuration corresponds to a unique set of moves that produce it, this subposet must contain all configurations reachable from $c$, so the join-irreducibles of the configuration poset correspond to elements of the move poset.
\end{proof}

\subsection{Related Problems and Counterexamples}

Now that we have shown that the ``endgame'' of chip-firing on the 1D grid forms a distributive lattice, it is natural to wonder whether certain stronger properties may be true.  We conclude with counterexamples to some of these stronger properties.

In particular, we ask if the entire configuration poset forms a distributive lattice.  In the odd case, we can show not only that this is not true, but that it is not true even during the corresponding endgame of the process.  It is already known to be upper locally distributive (ULD), but distributivity does not hold for sufficiently large $n$.  In the odd case, there are counterexamples to the distributivity condition at the last step of the process starting at $n=5$.  In the even case, the endgame is distributive, but earlier parts of the process are not starting at $n=8$.

For $n=5$, we have a counterexample to distributivity that takes place just one move from the end of the process.  In the following chip configurations, the underlined number represents the number of chips at the origin, with the numbers to the left and right representing the numbers of chips at corresponding sites away from the origin.

\begin{equation}
\begin{array}{r@=l}
    x&10\underline{3}01\\
    y&11\underline{0}3\\
    z&3\underline{0}11\\
    y \vee z&2\underline{1}2\\
    x \wedge (y \vee z)&10\underline{3}01\\
    x \wedge y&11\underline{1}11\\
    x \wedge z&11\underline{1}11\\
    (x \wedge y) \vee (x \wedge z)&11\underline{1}11
    
\end{array}
\end{equation}

Note that in the above example, $x$, $y$, and $z$ are all one firing move away from completion, so distributivity does not apply even in the endgame of odd labeled chip-firing.

While the odd case fails to exhibit distributivity even at the end of the process, even the even case is not distributive for sufficiently large $n$.  The smallest counterexample is $n=8$:

\begin{equation}
\begin{array}{r@=l}
    x&20\underline{3}21\\
    y&21\underline{0}5\\
    z&13\underline{0}31\\
    y \vee z&12\underline{1}4\\
    x \wedge (y \vee z)&20\underline{3}21\\
    x \wedge y&21\underline{1}31\\
    x \wedge z&21\underline{1}31\\
    (x \wedge y) \vee (x \wedge z)&21\underline{1}31
    
\end{array}
\end{equation}

This counterexample appears earlier on ($x$, $y$, and $z$ all appear after 8 moves in a 30 move firing sequence).

\begin{figure}
\centering
\begin{tikzpicture}[scale=0.4]

\node at (0,2) {Sites};
\node at (-1,1) {-1};
\node at (0,1) {0};
\node at (1,1) {1};

\filldraw[black](0,0) circle(2 pt);
\filldraw[black](0,-1) circle(2 pt);
\filldraw[black](0,-2) circle(2 pt);
\filldraw[black](1,-2) circle(2 pt);
\filldraw[black](-1,-2) circle(2 pt);

\draw[black](0,-1) -- (1,-2);
\draw[black](0,-1) -- (-1,-2);
\draw[black](0,0) -- (0,-2);

\node at (10,2) {Sites};
\node at (7,1) {-3};
\node at (8,1) {-2};
\node at (9,1) {-1};
\node at (10,1) {0};
\node at (11,1) {1};
\node at (12,1) {2};
\node at (13,1) {3};

\filldraw[black](10,0) circle(2 pt);
\filldraw[black](10,-1) circle(2 pt);
\filldraw[black](10,-2) circle(2 pt);
\filldraw[black](10,-3) circle(2 pt);
\filldraw[black](10,-4) circle(2 pt);
\filldraw[black](10,-5) circle(2 pt);
\filldraw[black](10,-6) circle(2 pt);
\filldraw[black](10,-8) circle(2 pt);
\filldraw[black](10,-10) circle(2 pt);
\filldraw[black](10,-12) circle(2 pt);
\filldraw[black](11,-2) circle(2 pt);
\filldraw[black](9,-2) circle(2 pt);
\filldraw[black](11,-4) circle(2 pt);
\filldraw[black](9,-4) circle(2 pt);
\filldraw[black](11,-5) circle(2 pt);
\filldraw[black](9,-5) circle(2 pt);
\filldraw[black](11,-7) circle(2 pt);
\filldraw[black](9,-7) circle(2 pt);
\filldraw[black](11,-9) circle(2 pt);
\filldraw[black](9,-9) circle(2 pt);
\filldraw[black](11,-11) circle(2 pt);
\filldraw[black](9,-11) circle(2 pt);
\filldraw[black](12,-5) circle(2 pt);
\filldraw[black](8,-5) circle(2 pt);
\filldraw[black](12,-8) circle(2 pt);
\filldraw[black](8,-8) circle(2 pt);
\filldraw[black](12,-10) circle(2 pt);
\filldraw[black](8,-10) circle(2 pt);
\filldraw[black](13,-9) circle(2 pt);
\filldraw[black](7,-9) circle(2 pt);

\draw[black](10,-1) -- (11,-2);
\draw[black](10,-1) -- (9,-2);
\draw[black](10,-3) -- (11,-4);
\draw[black](10,-3) -- (9,-4);
\draw[black](10,-4) -- (11,-5);
\draw[black](10,-4) -- (9,-5);
\draw[black](10,-6) -- (11,-7);
\draw[black](10,-6) -- (9,-7);
\draw[black](10,-8) -- (11,-9);
\draw[black](10,-8) -- (9,-9);
\draw[black](10,-10) -- (11,-11);
\draw[black](10,-10) -- (9,-11);
\draw[black](11,-4) -- (12,-5);
\draw[black](9,-4) -- (8,-5);
\draw[black](11,-7) -- (12,-8);
\draw[black](9,-7) -- (8,-8);
\draw[black](11,-9) -- (12,-10);
\draw[black](9,-9) -- (8,-10);
\draw[black](12,-8) -- (13,-9);
\draw[black](8,-8) -- (7,-9);
\draw[black](11,-2) -- (10,-5);
\draw[black](9,-2) -- (10,-5);
\draw[black](11,-5) -- (10,-6);
\draw[black](9,-5) -- (10,-6);
\draw[black](11,-7) -- (10,-8);
\draw[black](9,-7) -- (10,-8);
\draw[black](11,-9) -- (10,-10);
\draw[black](9,-9) -- (10,-10);
\draw[black](11,-11) -- (10,-12);
\draw[black](9,-11) -- (10,-12);
\draw[black](12,-5) -- (11,-7);
\draw[black](8,-5) -- (9,-7);
\draw[black](12,-8) -- (11,-9);
\draw[black](8,-8) -- (9,-9);
\draw[black](12,-10) -- (11,-11);
\draw[black](8,-10) -- (9,-11);
\draw[black](13,-9) -- (12,-10);
\draw[black](7,-9) -- (8,-10);
\draw[black](10,0) -- (10,-6);
\draw[black](11,-2) -- (11,-5);
\draw[black](9,-2) -- (9,-5);

\end{tikzpicture}
\caption{Move posets for $n=5$ and $n=8$, respectively the first odd and even cases in which many counterexamples start to appear.}
  \label{fign58}
\setlength{\belowcaptionskip}{-10pt}
\end{figure}
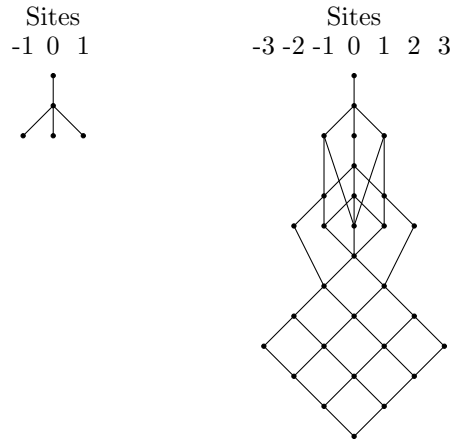

We also examined whether there is a bijection between firing histories and linear orderings of the firing move poset for the entire process.  While this is true for the endgame of labeled chip-firing for $n$ even, it is not true over the entire process.  There is a clear injection from firing histories to linear orderings of the move poset, but there are possible orderings that don't correspond to valid firing sequences.

The issue that arises is that the validity of a firing move at a certain time depends on firing moves that occur at both of the site's neighbors.  Statements of the form ``this move at site $k$ must take place after either this move at site $k-1$ or this move at site $k+1$'' occur regularly in the chip-firing process but are not captured by the poset.  As a result, there are orderings of the poset that do not correspond to legal firing sequences starting at the odd case $n=5$ and the even case $n=8$.

Observing the hasse diagrams from Figure \ref{fign58}, we see that the poset for $n=5$ allows for a linear ordering with firing moves at sites $(0,0,0,1,-1)$.  This, however, would result in a negative number of chips appearing at site 0 after the third firing move, meaning that it is not a valid firing order.

Similarly, the poset for $n=8$ would allow for a linear ordering in which the first 5 firing moves all take place at the origin.  This would also result in a negative number of chips at the origin after the fifth move, so there are also linear orderings that don't correspond to valid firing sequences in this case.

\bibliographystyle{plain}

\bibliography{References}

\end{document}